\documentclass[reqno,tbtags,a4paper]{amsart}

\usepackage{amsmath,amssymb,amsthm,amsfonts,amscd,array,graphicx}
\usepackage{mathrsfs}

\newcommand{\F}{F}

\newcommand{\R}{\mathbb{R}}

\newcommand{\Q}{\mathbb{Q}}
\newcommand{\K}{\mathbb{K}}
\renewcommand{\L}{\mathbb{L}}
\newcommand{\Z}{\mathbb{Z}}

\newcommand{\Int}{\mathop{\mathrm{Int}}\nolimits}

\newtheorem{theorem}{Theorem}[section]
\newtheorem{propos}[theorem]{Proposition}

\newtheorem{lem}[theorem]{Lemma}

\theoremstyle{definition}
\newtheorem{defin}[theorem]{Definition}

\newtheorem{quest}[theorem]{Question}

\author{Alexander A. Gaifullin, Sergey A. Gaifullin}

\thanks{The first author was partially supported by the Russian Foundation for Basic Research (projects 12-01-31444 and 13-01-12469), by a grant of the President of the Russian Federation (project MD-4458.2012.1), by a grant of the Government of the Russian Federation (project 11.G34.31.0053),  by a program of the Branch of Mathematical Sciences of the Russian Academy of Sciences, and by a grant from Dmitry Zimin's ``Dynasty'' foundation.}

\thanks{The second author was partially supported by the Russian Foundation for Basic Research (projects 12-01-31342 and 12-01-00704) and by the Ministry of Education and Science of the Russian Federation (project 8214)}



\title{Deformations of period lattices of flexible polyhedral surfaces}

\date{}

\address{Alexander A. Gaifullin:\newline
Steklov Mathematical Institute, Moscow, Russia\newline
Lomonosov Moscow State University, Moscow, Russia\newline 
Kharkevich Institute for Information Transmission Problems, Moscow, Russia\newline 
Demidov Yaroslavl State University, Yaroslavl, Russia}
\email{agaif@mi.ras.ru}

\address{Sergey A. Gaifullin:\newline
Lomonosov Moscow State University, Moscow, Russia\newline 
Higher School of Economics, Moscow, Russia}
\email{sgayf@yandex.ru}

\begin{document}

\begin{abstract}
In the end of the 19th century Bricard discovered a phenomenon of flexible polyhedra, that is, polyhedra with rigid faces and hinges at edges that admit non-trivial flexes. One of the most important results in this field is a theorem of Sabitov asserting that the volume of a flexible polyhedron is constant during the flexion.
In this paper we study flexible polyhedral surfaces in~$\R^3$  two-periodic with respect to  translations by two non-colinear vectors that can vary continuously during the flexion. The main result is that the period lattice of a flexible two-periodic surface homeomorphic to a plane cannot have two degrees of freedom.

\end{abstract}

\maketitle
\section{Introduction}
We denote by~$\R^3$ the $3$-dimensional Euclidean space.
Let $S\subset\R^3$ be a polyhedral surface that is homeomorphic to a plane and is situated ``near'' the horizontal plane~$\R^2$. Suppose that $S$ has rigid faces and hinges at edges. This means that $S$ is allowed to flex so that the faces remain congruent to themselves, while the dihedral angles at edges change continuously. For example, let us take for $S$ the horizontal plane $z=0$ with hinges at lines $y=k$, $k\in\Z$. Then $S$ can be shrunk in the direction of the $x$-axis as it is shown in Figure~\ref{fig_grebenka}(a). Now, take for $S$ the horizontal plane with hinges both at lines $y=k$ and $x=k$, $k\in\Z$. Then $S$ can be shrunk either in the direction of the $x$-axis, or in the direction of the $y$-axis. However, once we have already started to shrink~$S$ in  the direction of the $x$-axis, we are not able any more to shrink it in the direction of the $y$-axis, and vice versa, see Figure~\ref{fig_grebenka}(b). 
In this paper we consider the following natural question:

\textit{Is it possible to construct a polyhedral surface $S\subset\R^3$ that can be shrunk near the horizontal plane independently in two different directions? In other words, is it possible that there exists a two-parametric flexion of~$S$ such that varying the first parameter, we shrink $S$ in the first direction, and varying the second parameter, we shrink~$S$ in the second direction?}

The authors are indebted to an architect Sergei Kolchin who suggested this question to them.

\begin{figure}
\unitlength=.4mm
\begin{picture}(236,241)
\put(115,0){\textit{b}}
\put(115,173){\textit{a}}

\put(52,44){%
\begin{picture}(0,0)
\multiput(-52,-24)(8,12){5}{\line(1,0){72}}
\multiput(-52,-24)(12,0){7}{\line(2,3){32}}
\end{picture}%
}

\put(100,44){\vector(1,0){36}}

\put(184,44){%
\begin{picture}(0,0)
\multiput(-46,-27)(8,12){5}{%
\begin{picture}(0,0)
\multiput(0,0)(20,0){3}{\line(5,3){10}}
\multiput(10,6)(20,0){3}{\line(5,-3){10}}
\end{picture}%
}

\multiput(-46,-27)(20,0){4}{\line(2,3){32}}
\multiput(-36,-21)(20,0){3}{\line(2,3){32}}
\end{picture}%
}

\put(156,128){\vector(1,0){36}}
\put(170,134){\line(2,-3){8}}
\put(170,122){\line(2,3){8}}

\put(74,77){\vector(2,3){12}}

\put(206,77){\vector(2,3){12}}
\put(205,86){\line(1,0){14}}
\put(209,92){\line(1,-2){6}}

\put(108,128){%
\begin{picture}(0,0)
\multiput(-48,-21)(12,18){3}{\line(1,0){72}}
\multiput(-42,-6)(12,18){2}{\line(1,0){72}}

\multiput(-48,-21)(12,0){7}{%
\begin{picture}(0,0)
\multiput(0,0)(12,18){2}{\line(2,5){6}}
\multiput(6,15)(12,18){2}{\qbezier(0,0)(3,1.5)(6,3)}
\end{picture}%
}

\end{picture}%
}

\put(52,217){%
\begin{picture}(0,0)
\put(-52,-24){\line(1,0){72}}
\put(52,24){\line(-1,0){72}}
\multiput(-52,-24)(12,0){7}{\line(2,3){32}}
\end{picture}%
}

\put(100,217){\vector(1,0){36}}

\put(184,217){%
\begin{picture}(0,0)
\multiput(-46,-27)(32,48){2}{%
\begin{picture}(0,0)
\multiput(0,0)(20,0){3}{\line(5,3){10}}
\multiput(10,6)(20,0){3}{\line(5,-3){10}}
\end{picture}%
}

\multiput(-46,-27)(20,0){4}{\line(2,3){32}}
\multiput(-36,-21)(20,0){3}{\line(2,3){32}}
\end{picture}%
}

\end{picture}
\caption{Simplest examples of flexible polyhedral surfaces}\label{fig_grebenka}
\end{figure}
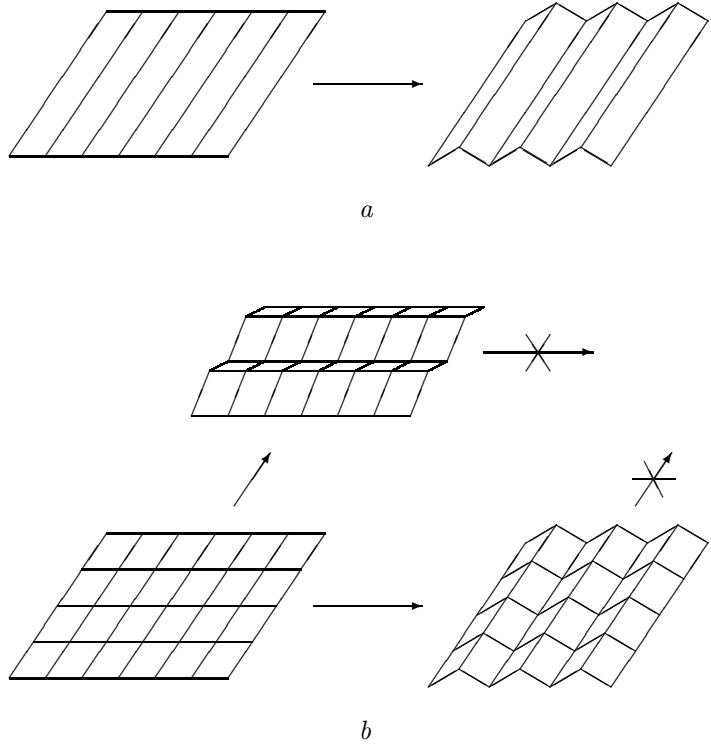

 In this paper we focus on the case of a surface $S$ invariant under the translations by two non-colinear vectors~$a$ and~$b$. The quotient of~$S$ by the action of the lattice~$\Lambda$ generated by~$a$ and~$b$ is homotopy equivalent to the torus. On the other hand, $S/\Lambda$ is a two-dimensional manifold without boundary. Hence $S/\Lambda$ is homeomorphic to the torus. In particular, it is compact. Therefore the surface $S$ is automatically contained in a layer between two planes parallel to the vectors~$a$ and~$b$. By a rotation, we may achieve that the distances from the points of~$S$ to the standard horizontal plane~$\R^2$ are uniformly  bounded.  

Now we proceed with a more rigorous statement of the problem.
A \textit{$2$-periodic polyhedral surface\/} is a triple~$(S,a,b)$ such that $S\subset\R^3$ is a polyhedral surface (i.\,e. a surface without boundary glued out of polygons along their sides), and $a,b\in\R^3$ are non-colinear vectors such that

1) $S$ is  invariant under the translations by~$a$ and~$b$,
 
2) The action of the lattice $\Lambda$ generated by~$a$ and~$b$ on the surface~$S$ is \textit{cocompact\/}, i.\,e., the quotient~$S/\Lambda$ is compact.

As it has been mentioned above, the second condition holds automatically if $S$ is homeomorphic to a plane.

We allow the polyhedral surface $S$ to flex so that it remains $2$-periodic and the period lattice changes continuously. So by definition, a \textit{flex
of a 2-periodic polyhedral surface\/}~$(S,a,b)$ is a continuous deformation $(S(t),a(t),b(t))$, $t\in[0,1]$, $S(0)=S$, $a(0)=a$, $b(0)=b$, such that $S(t)$ is a flex of~$S$, and for each~$t$, the triple $(S(t),a(t),b(t))$ is a $2$-periodic polyhedral surface.

Up to rotations of~$\R^3$ the period lattice~$\Lambda$ of the $2$-periodic polyhedral surface $(S,a,b)$ is determined by the Gram matrix~$G$ of the vectors~$a$ and~$b$. We are interested in the deformations of the period lattice~$\Lambda$ under flexes. But we would like to neglect rotations. So we are interested in the deformations of the Gram matrix~$G$ induced by flexions of~$(S,a,b)$.
The main result of the present paper says that if $S$ is homeomorphic to a plane, then only one-parametric deformations of~$G$ may occur. This means that we cannot find a two-parametric flexion that yields a truly two-parametric deformation of~$G$.
This can be formulated more formally as follows. Consider all possible $2$-periodic polyhedral surfaces $(S',a',b')$ that can be obtained from $(S,a,b)$ by  flexions.  Let $\mathcal{G}=\mathcal{G}(S,a,b)$ be the set of all matrices~$G'$ that appear as the Gram matrices of the vectors $a',b'$ for such $2$-periodic surfaces $(S',a',b')$.

\begin{theorem}\label{theorem_main}
Let $(S,a,b)$ be a $2$-periodic polyhedral surface homeomorphic to a plane. Then the set $\mathcal G(S,a,b)$ is contained in a one-dimensional real affine algebraic variety.
\end{theorem}

This theorem can be generalized to the case of a non-embedded polyhedral surface, see Theorem~\ref{main_t}.

The interest to results of such kind originated from the famous result of Sabitov~\cite{Sab96},~\cite{Sab98} that the volume of a (compact) flexible polyhedron in~$\R^3$ remains constant under the flexion. Sabitov obtained this result by proving that the volume~$V$ of every simplicial polyhedron satisfies a polynomial relation of the form
\begin{equation}\label{eq_V}
V^{2N}+a_1(\ell)V^{2N-2}+\cdots+a_N(\ell)=0,
\end{equation}
where $a_i(\ell)$ are polynomials with rational coefficients in the squares of the edge lengths of the polyhedron. (The number~$N$ and  the polynomials~$a_i$ depend on the combinatorial structure of the polyhedron.) The same result for polyhedra of dimensions $n\ge 4$ has recently been obtained by one of the authors~\cite{Gai11},~\cite{Gai12}. 

A natural question is which other invariants of polyhedra satisfy polynomial relations of the form~\eqref{eq_V}. In the periodic setting, the natural invariants are the coefficients $g_{11}$, $g_{12}$, and~$g_{22}$ of the Gram matrix of the period lattice. Though we cannot obtain a polynomial relation of the form~\eqref{eq_V} for any of these coefficients, we shall prove that the three coefficients $g_{11}$, $g_{12}$, and~$g_{22}$, and the set~$\ell$ of the squares of edge lengths are subject to two polynomial relations such that for any given~$\ell$, these relations yield an affine variety of dimension not greater than~$1$. 

Some ideas of our proof of Theorem~\ref{theorem_main} are inspired by Sabitov's theorem, and especially by another proof of Sabitov's theorem obtained by Connelly, Sabitov, and Walz~\cite{CSW97}. The main tool is theory of places (see section~\ref{section_places}). It is standard to use  places to prove that certain element is integral over the given ring. Our situation is more difficult, namely, we need to prove that among the three given elements~$g_{11}$, $g_{12}$, and~$g_{22}$ there exist at least two independent relations over the given ring. 


Notice that we can easily construct $2$-periodic polyhedral surfaces with arbitrarily large number of degrees of freedom  taking connected sums of the plane with flexible polyhedra (see Figure~\ref{fig_flex}). Such flexes can be called ``local'', since they do not change the surface off some small disks, and, in particular, do not affect the period lattice. Theorem~\ref{theorem_main} describes the ``global behavior'' of  flexible $2$-periodic polyhedral surfaces.   

\begin{figure}
\input{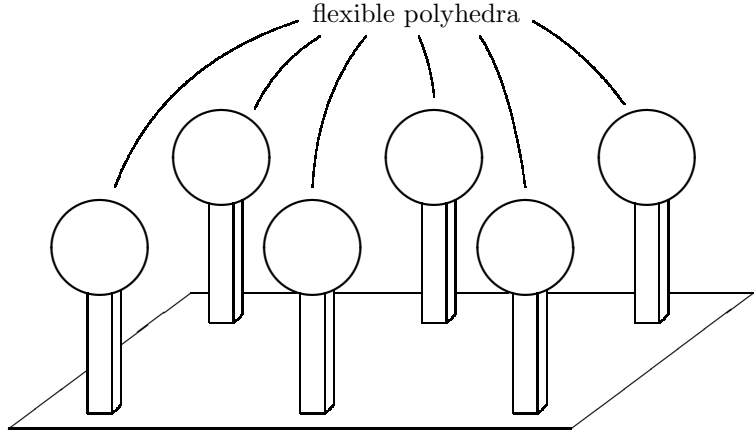}
\caption{2-periodic connected sum of the plane with flexible polyhedra}\label{fig_flex}
\end{figure}

Now, we would like to formulate two natural questions, which remain open.


\begin{quest}
Does the assertion of Theorem~\ref{theorem_main} holds for an arbitrary $2$-periodic polyhedral surface, not necessarily homeomorphic to a plane? 
\end{quest}

\begin{quest}
Let $(S,a_1,\ldots,a_k)$ be a $k$-dimensional $k$-periodic polyhedral surface in~$\R^n$. What is the maximal possible dimension of~$\mathcal{G}(S,a_1,\ldots,a_k)$?
\end{quest}

The authors are grateful to V.\,M.~Buchstaber and I.\,Kh.~Sabitov for useful discussions.

\section{Examples}\label{section_examples}

The set $\mathcal{G}(S,a,b)$ depends not only on the polyhedral surface~$S$, but also on the vectors~$a$ and~$b$. If $\tilde a$, $\tilde b$ is another basis of the same lattice~$\Lambda$, $(\tilde a,\tilde b)=(a,b)C$, then the usual formula
$\widetilde{G}=C^TGC$ provides a canonical affine isomorphism of the sets~$\mathcal{G}(S,a,b)$ and $\mathcal{G}(S,\tilde a,\tilde b)$. We shall identify them and denote this set by~$\mathcal{G}(S,\Lambda)$.

Now, let us consider a sublattice~$\widetilde{\Lambda}$ of the lattice~$\Lambda$. Let $a,b$ be a basis of~$\Lambda$ and let $\tilde a$, $\tilde b$ be a basis of~$\widetilde{\Lambda}$. Then the same formula $\widetilde{G}=C^TGC$ yields the affine embedding $\mathcal{G}(S,\Lambda)\hookrightarrow\mathcal{G}(S,\widetilde{\Lambda})$. Nevertheless, this embedding is not necessarily an isomorphism, since the polyhedral surface $S$ may have flexes~$S_t$ such that $S_t$ remains $2$-periodic with respect to a continuous deformation~$\widetilde{\Lambda}_t$ of the lattice~$\widetilde{\Lambda}$, but does not remain $2$-periodic with respect to any continuous deformation of the lattice~$\Lambda$.

Let us give an example. Consider a plane and fix a point~$p$ in it and two non-collinear vectors $a$ and $b$ parallel to it. Let us divide this plane into triangles by the straight lines parallel to $a$ through the points $p+kb$, $k\in\mathbb{Z}$, the straight lines parallel to $b$ through the points $p+ka$, $k\in \mathbb{Z}$, and the straight lines parallel to $a-b$ through the points $p+ka$, $k\in \mathbb{Z}$. Let $S$ be the polyhedral surface consisting of all these triangles. It is easy to see that the $2$-periodic polyhedral surface $(S, \Lambda)$ is not flexible, where $\Lambda=\langle a,b\rangle$. However, the surface $S$ will become flexible if we replace $\Lambda$ by a sublattice of it.

If we consider the period lattice $\widetilde{\Lambda}=\langle a,2b\rangle$, then the $2$-periodic polyhedral surface $(S, \widetilde{\Lambda})$ admits only one type of flexion. We can shrink it in the direction orthogonal to~$a$. Then for the Gram matrix of the basis of~$\widetilde{\Lambda}$, we have $g_{11}=const$, $g_{12}=const$,  and $g_{22}$ varies from~$\frac{4(a,b)^2}{|a|^2}$ to~$4|b|^2$. Hence $\mathcal{G}(S,\widetilde{\Lambda})$ is a segment.

If we consider the period lattice $\widehat{\Lambda}=\langle 2a,2b\rangle$, then we can  shrink $(S,\widehat{\Lambda})$ in the direction orthogonal to $a$, in the direction orthogonal to $b$, or in the direction orthogonal to $a-b$. Hence $\mathcal{G}(S,\widehat{\Lambda})$ contains three segments. (We do not claim that  $\mathcal{G}(S,\widehat{\Lambda})$ consists only of these three segments.)

Now, let us give an example of $(S,\Lambda)$ such that the Zariski closure of $\mathcal{G}(S,\Lambda)$ has an irreducible component which is not a straight line. 
This example is a well-known flexible surface consisting of parallelograms (see Figure~\ref{fig_miura}). This surface can be folded so that full lines are mountain foldings and dashed lines are valley foldings, cf.~\cite{Sta11}. Such folding was known in ancient Japanese origami technique and is called the Miura-ori folding after Miura who suggested to apply this type of folding for solar panels. The polyhedral surface~$S$ shown in Figure~\ref{fig_miura} will be considered with the period lattice   
$\Lambda=\langle a,b\rangle$.

\begin{figure}

\unitlength=0.4mm
\begin{picture}(250, 175)
\multiput(0,0)(72,0){3}%
{\begin{picture}(100,200)
\multiput(0,0)(8,12){3}%
{\line(2,3){7}}
\multiput(0,70)(8,12){3}%
{\line(2,3){7}}
\multiput(0,140)(8,12){3}%
{\line(2,3){7}}
\multiput(23,35)(-8,12){3}%
{\line(-2,3){7}}
\multiput(23,105)(-8,12){3}%
{\line(-2,3){7}}

\put(0,0){\line(1,0){36}}
\put(0,70){\line(1,0){36}}
\put(0,140){\line(1,0){36}}
\multiput(23,35)(13,0){3}%
{\line(1,0){10}}
\multiput(23,105)(13,0){3}%
{\line(1,0){10}}
\multiput(23,175)(13,0){3}%
{\line(1,0){10}}

\put(36,0){\line(2,3){23}}
\put(36,70){\line(2,3){23}}
\put(36,140){\line(2,3){23}}
\put(59,35){\line(-2,3){23}}
\put(59,105){\line(-2,3){23}}

\put(59,35){\line(1,0){36}}
\put(59,105){\line(1,0){36}}
\put(59,175){\line(1,0){36}}
\multiput(36,0)(13,0){3}%
{\line(1,0){10}}
\multiput(36,70)(13,0){3}%
{\line(1,0){10}}
\multiput(36,140)(13,0){3}%
{\line(1,0){10}}

\end{picture} }

\multiput(216,0)(8,12){3}%
{\line(2,3){7}}
\multiput(216,70)(8,12){3}%
{\line(2,3){7}}
\multiput(216,140)(8,12){3}%
{\line(2,3){7}}
\multiput(239,35)(-8,12){3}%
{\line(-2,3){7}}
\multiput(239,105)(-8,12){3}%
{\line(-2,3){7}}

{\thicklines \put(30,17){\vector(1,0){72}}
 \put(30,17){\vector(0,1){70}}
\put(102,17){\circle*{3}}
\put(30,17){\circle*{3}}
\put(30,87){\circle*{3}}
\put(22,50){$\boldsymbol{b}$}
\put(66,20){$\boldsymbol{a}$}
}

\end{picture}

\caption{The Miura-ori folding}\label{fig_miura}

\end{figure}
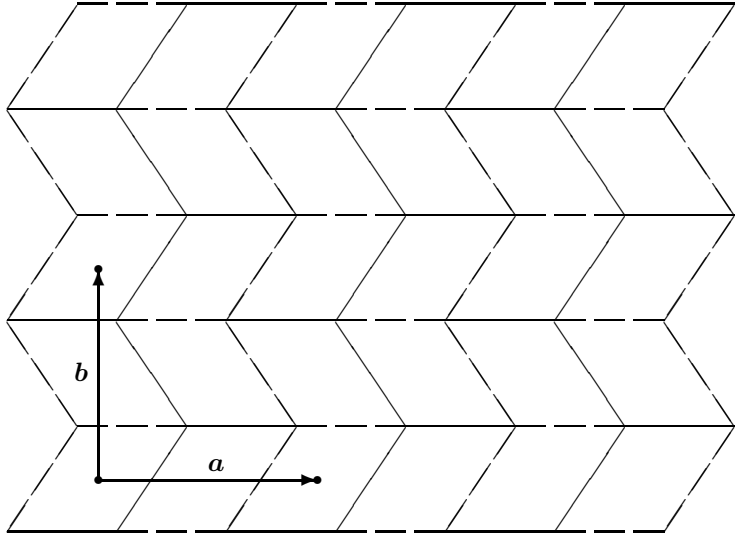

 Since the period lattice~$\Lambda$ changes during the flexion, we shall further denote the initial basis vectors~$a$ and~$b$ by~$a_0$ and~$b_0$ respectively, and we shall use the notation~$a$ and~$b$ for the period vectors after the flex. By a rotation of~$\R^3$ we may achieve that the vectors~$a$ and~$b$ are parallel to the vectors~$a_0$ and~$b_0$ respectively. Take the orthonormal basis $e_1=\frac{a_0}{|a_0|}$, $e_2=\frac{b_0}{|b_0|}$, $e_3$ in~$\R^3$. Let $\alpha$ be the acute angle of a parallelogram of~$S$. Then the initial sides of a parallelogram of~$S$ with acute lower-left angle are the vectors $\xi_0=\bigl(\frac{|a_0|}{2},0,0\bigr)$ and $\eta_0=\bigl(\frac{|b_0|}{2}\cot\alpha,\frac{|b_0|}{2}, 0\bigr)$. During the flexion the sides~$\xi$ and~$\eta$ of this parallelogram remain parallel to the planes $\langle e_1,e_3\rangle$ and $\langle e_1,e_2\rangle$ respectively, $\xi=\bigl(\frac{|a|}{2},0,z\bigr)$ and $\eta=\bigl(x,\frac{|b|}{2}, 0\bigr)$.
Since $(\eta,\eta)=(\eta_0,\eta_0)$ and $(\xi,\eta)=(\xi_0,\eta_0)$, we have
\begin{gather*}
x^2+\frac{|b|^2}{4}=\frac{|b_0|^2}{4\sin^2\alpha}\,,\\
|a|x=\frac12|a_0||b_0|\cot\alpha.
\end{gather*}
Hence,
$$
|a|^2|b|^2\sin^2\alpha-|a|^2|b_0|^2+|a_0|^2|b_0|^2\cos^2\alpha=0.
$$
Therefore, the Gram matrix elements of the period vectors~$a$ and~$b$ satisfy the system of equations
$$
\left\{
\begin{aligned}
&g_{11}g_{22}\sin^2\alpha-g_{11}|b_0|^2+|a_0|^2|b_0|^2\cos^2\alpha=0,\\
&g_{12}=0.
\end{aligned}
\right.
$$
Thus, $\mathcal{G}(S,\Lambda)$ contains a segment of a hyperbola.
 
Certainly,  $\mathcal{G}(S,\Lambda)$ also contains a straight segment corresponding to the shrinking in the direction orthogonal to $a$.

\section{Notation}\label{section_places}

We recall some basic facts on places of fields. For more detailed information and proofs see~\cite{Lan72},~\cite{Pak10}.

Suppose $L$ and $F$ are fields.

\begin{defin}
A mapping $\varphi\colon L\rightarrow F\cup\{\infty\}$ is called a {\it place} of~$L$ into~$F$ if for all $x,y\in L$ we have:

1) $\varphi(x+y)=\varphi(x)+\varphi(y),$

2) $\varphi(xy)=\varphi(x)\varphi(y)$,

3) $\varphi(1)=1$,

\noindent{where it is understood that for $z\in F$ (called {\it finite} $z$), $z\pm\infty = \infty\cdot\infty=\frac{1}{0}=\infty$, and if $z\neq 0$, $z\cdot\infty=\infty$. The expressions $\frac{0}{0}$, $\frac{\infty}{\infty}$, $0\cdot\infty$, $\infty\pm\infty$ are not defined. And it is also understood that 1) and 2) only hold if the right hand side is defined.}
\end{defin} 

It is clear that the restriction of a place to a subfield is a place, and the composition of places is a place. (Taking the composition of places $\varphi\colon L\to F\cup\{\infty\}$ and $\psi\colon F\to E\cup\{\infty\}$, we always put $\psi(\infty)=\infty$.)

\begin{defin}
A mapping of a simplicial complex~$K_1$ to a simplicial complex~$K_2$ is called \textit{simplicial\/} if it maps every simplex of~$K_1$ linearly onto a simplex of~$K_2$. An action of a discrete group~$\Lambda$ on a simplicial complex~$K$ is called \textit{simplicial\/} if every element of~$\Lambda$ acts by a simplicial mapping $K\to K$.
\end{defin}

Let $K$ be a simplicial complex homeomorphic to a plane with a free simplicial action of the group $\Lambda=\langle\alpha,\beta\rangle\cong\Z^2$.
We shall use the additive notation for the group~$\Lambda$. For an element $\lambda\in\Lambda$, we denote by~$T_{\lambda}$ the corresponding simplicial automorphism of~$K$. Then we have $T_{\lambda}T_{\mu}=T_{\lambda+\mu}$.

\begin{defin} Consider a  mapping $\theta\colon K\rightarrow \mathbb{R}^3$ linear on simplices of~$K$ and equivariant with respect to an action of~$\Lambda$ on~$\R^3$ such that $\alpha$ and~$\beta$ act by translations by some vectors~$a=a(\theta)$ and~$b=b(\theta)$. Then the pair~$(K,\theta)$ is called  a {\it 2-periodic polyhedral surface}. If $\theta$ is injective,  the polyhedral surface~$(K,\theta)$ is called \textit{embedded}.
\end{defin}

Let $v_1,\ldots,v_n$ be representatives of all  $\Lambda$-orbits of vertices of~$K$. Consider the field $$\L=\L(K,\Lambda)=\Q(x_{\alpha},y_{\alpha},z_{\alpha},x_{\beta},y_{\beta},z_{\beta}, x_{v_1},y_{v_1},z_{v_1},\ldots, x_{v_n},y_{v_n},z_{v_n}),$$ where $x_{\alpha},y_{\alpha},z_{\alpha},x_{\beta},y_{\beta},z_{\beta}, x_{v_1},y_{v_1},z_{v_1},\ldots, x_{v_n},y_{v_n},z_{v_n}$ are independent  over $\Q$ variables. For $\lambda=m\alpha+k\beta$, we denote $x_{\lambda}=mx_{\alpha}+kx_{\beta}$, $y_{\lambda}=my_{\alpha}+ky_{\beta}$, $z_{\lambda}=mz_{\alpha}+kz_{\beta}$. For $u=T_{\lambda}(v_i)$, we denote $x_u=x_\lambda+x_{v_i}$, $y_u=y_{\lambda}+y_{v_i}$, $z_u=z_{\lambda}+z_{v_i}$. 

If we have a 2-periodic polyhedral surface $(K,\theta)$, then we obtain the specialization homomorphism
$$
\tau_{\theta}\colon\Q[x_{\alpha},y_{\alpha},z_{\alpha},x_{\beta},y_{\beta},z_{\beta}, x_{v_1},y_{v_1},z_{v_1},\ldots, x_{v_n},y_{v_n},z_{v_n}]\to\R
$$
that takes $x_u$, $y_u$ and $z_u$ to the coordinates of the point~$\theta(u)$ for every vertex~$u$, takes $x_{\alpha}$, $y_{\alpha}$, and $z_{\alpha}$ to the coordinates of the vector~$a$, and takes $x_{\beta}$, $y_{\beta}$, and $z_{\beta}$ to the coordinates of the vector~$b$. 

We shall conveniently identify a vertex $v$ of~$K$ with the point $(x_v,y_v,z_v)\in\L^3$. We shall also identify an element~$\lambda\in\Lambda$ with the vector $(x_\lambda,y_\lambda,z_\lambda)\in\L^3$. Then $T_{\lambda}(v)$ is identified with~$v+\lambda$. Thus we obtain the linear embedding $\Lambda\subset\L^3$. 

We endow the vector space~$\L^3$ with the standard inner product given by
$$
(\xi,\eta)=\xi_1\eta_1+\xi_2\eta_2+\xi_3\eta_3, \qquad\xi=(\xi_1,\xi_2,\xi_3),\,\eta=(\eta_1,\eta_2,\eta_3).
$$ 
We put, 
$$g_{11}=(\alpha,\alpha),\  g_{12}=(\alpha,\beta),\  g_{22}=(\beta,\beta)\in \L.$$
Let $u$ and $v$ be two vertices of $K$. We put $\ell_{uv}=(v-u,v-u)\in\L$.

The homomorphism~$\tau_\theta$  takes $g_{11}$, $g_{12}$, and  $g_{22}$ to the elements of the Gram matrix of the vectors $a(\theta)$ and~$b(\theta)$, and takes $\ell_{uv}$ to the square of the distance between $\theta(u)$ and $\theta(v)$.

Let $R=R(K,\Lambda)$ be a $\Q$-subalgebra of the field $\L$ generated by all $\ell_{uv}$ such that $[uv]$ is an edge of~$K$.

\section{Main result}

Fix a set of numbers $l=\{l_{uv}\},\ l_{uv}=l_{vu}\in\R$, where $[uv]$ runs over all edges of $K$. Consider all possible $2$-periodic polyhedral surfaces $(K,\theta)$ with the set of the squares of edge lengths equal to $l$.  Let $\mathcal{G}=\mathcal{G}(K,\Lambda,l)$ be the set of all matrices~$G$ that appear as the Gram matrices of the vectors $a(\theta),b(\theta)$ for such $2$-periodic surfaces $(K,\theta)$. (If $l$ cannot be realized as the set of the squares of the edge lengths of a polyhedral surface, then the set $\mathcal{G}(K,\Lambda,l)$ is empty.) The set~$\mathcal{G}$ is contained in the $3$-dimensional affine space~$\R^3$ with coordinates~$g_{11},g_{12},g_{22}$.

The following theorem is a strengthened version of Theorem~\ref{theorem_main} for not necessarily embedded polyhedral surfaces.

\begin{theorem}\label{main_t}
Let $K$ be a simplicial complex homeomorphic to~$\R^2$ with a free simplicial action of the group $\Lambda=\langle\alpha,\beta\rangle\cong\Z^2$. Then for each set of numbers $l=\{l_{uv}\},\ l_{uv}=l_{vu}\in\R$, there is a one-dimensional real affine algebraic subvariety of~$\R^3$ containing $\mathcal{G}(K,\Lambda,l)$.
\end{theorem}


To prove this theorem we shall study polynomial relations among the elements $g_{11}, g_{12},g_{22}\in\L=\L(K,\Lambda)$ with coefficients in the ring~$R=R(K,\Lambda)$. Any such polynomial relation has the form $f(g_{11},g_{12},g_{22})=0$, where $f\in R[X,Y,Z]$. The free $R$-algebra~$R[X,Y,Z]$ has a natural $\Z$-grading given by $\deg X=\deg Y=\deg Z=1$. For each $f\in R[X,Y,Z]$, we denote by~$\widehat{f}$ the homogeneous component of~$f$ of the maximal degree.

\begin{propos}\label{propos_main}
Let $K$ be a simplicial complex homeomorphic to~$\R^2$ with a free simplicial action of a group $\Lambda=\langle\alpha,\beta\rangle\cong\Z^2$, $\L=\L(K,\Lambda)$, and $R=R(K,\Lambda)$. Then the elements $g_{11},g_{12},g_{22}\in\L$ satisfy a system of two polynomial equations 
$$\left\{
\begin{aligned}
f(g_{11},g_{12},g_{22})&=0,\\
h(g_{11},g_{12},g_{22})&=0
\end{aligned}
\right.
$$
with coefficients in~$R$ such that $\widehat{f}$ and $\widehat{h}$ have coefficients in $\Q$, and are coprime in~$\Q [X,Y,Z]$. 
\end{propos}

\begin{proof}
Let us formulate a key lemma.

\begin{lem}\label{lem_main}
Let $K$ be a simplicial complex homeomorphic to~$\R^2$ with a free simplicial action of a group $\Lambda\cong\Z^2$, and let $\L=\L(K,\Lambda)$. Let $\varphi:\L\to\F\cup\{\infty\}$ be a place such that $\mathop{\mathrm{char}}\nolimits\F\ne 2$ and $\varphi(\ell_{uv})\ne\infty$ for all edges~$[uv]$ of~$K$. Then there exists a basis $\lambda,\mu$ of~$\Lambda$ such that $\varphi$ is finite on the inner products~$(\lambda,\lambda)$ and~$(\lambda,\mu)$.
\end{lem}

This lemma will be proved in the next section, and now we shall use it to prove Proposition~\ref{propos_main}. 

Recall that an element $\lambda\in\Lambda$ is called \textit{primitive\/} if it does not have the form $q\mu$ for an integer $q>1$ and $\mu\in\Lambda$. An element $m\alpha+k\beta$ is primitive if and only if $m$ and~$k$ are coprime.

Let us construct two $\Q$-subalgebras $R_1$ and $R_2$ of~$\L$. The algebra $R_1$ is obtained by adjoining to $R$ the inverted inner squares of all primitive elements of $\Lambda$:
$$
R_1=R\left[\left.\frac{1}{(\lambda, \lambda)}\,\right| \lambda\in\Lambda \text{ is a primitive element}\right].
$$  
The algebra $R_2$ is obtained by adjoining to $R$ the inverted inner products of all pairs of non-proportional primitive elements of $\Lambda$:
$$
R_2=R\left[\left.\frac{1}{(\lambda, \mu)}\,\right| \lambda\neq \pm\mu \text{ are primitive elements of } \Lambda \right].
$$
Let us consider the ideal $$I_1=\left(\left.\frac{1}{(\lambda,\lambda)}\,\right| \lambda\in\Lambda \text{ is a primitive element}\right)\triangleleft R_1.$$ There are two cases, $I_1=R_1$ and $I_1\neq R_1$.

If $I_1\neq R_1$, then there exists a maximal ideal $\mathfrak{m}_1\supset I_1$. Consider the field $F=R_1/\mathfrak{m}_1$. Since $R_1$ contains~$\Q$, we have $\mathop{\mathrm{char}}\F=0$. The quotient homomorphism $\varphi\colon R_1\rightarrow F$ can be extended to a place $\varphi\colon \L\rightarrow \overline{F}\cup \{\infty\}$, where $ \overline{F}$ is the algebraic closure of~$\F$ (see~\cite[Ch.~I, Thm. 1]{Lan72}). The place $\varphi$ is finite on $R_1$ and vanishes on~$\mathfrak{m}_1$. Hence $\varphi$ is finite on $R$ and infinite on $(\lambda,\lambda)$ for all primitive elements $\lambda\in \Lambda$. This contradicts Lemma \ref{lem_main}. Consequently, this is not the case.

If $I_1=R_1$, then $1\in I_1$. Hence 
$$
1=\sum_{i=1}^{p}\frac{r_i}{(\lambda_{i1},\lambda_{i1})(\lambda_{i2},\lambda_{i2})\dots(\lambda_{iq_i},\lambda_{iq_i})}\ ,
$$ 
where $r_i\in R$, $q_i\geq 1$, and $\lambda_{ij}$ are primitive elements of $\Lambda$. Multiplying both sides by the product of all denominators, and putting $\lambda_{ij}=m_{ij} \alpha+k_{ij} \beta$, $m_{ij},k_{ij}\in \mathbb{Z},$ we obtain an algebraic equation on~$g_{ij}$:

$$
\prod_{i=1}^{p}\prod_{j=1}^{q_i}(m_{ij}^2g_{11}+2m_{ij}k_{ij}g_{12}+k_{ij}^2g_{22})=r_1\prod_{i=2}^{p}\prod_{j=1}^{q_i}(m_{ij}^2g_{11}+2m_{ij}k_{ij}g_{12}+k_{ij}^2g_{22})+\ldots
$$

Taking all summands to the left-hand side, we obtain the relation of the form $f(g_{11},g_{12},g_{22})=0$. Since $q_i\ge 1$ for all~$i$, we see that 
$$
\widehat{f}(X,Y,Z)=\prod_{i=1}^{p}\prod_{j=1}^{q_i}(m_{ij}^2X+2m_{ij}k_{ij}Y+k_{ij}^2Z)=\prod_{i=1}^q(m_{i}^2X+2m_{i}k_{i}Y+k_{i}^2Z), 
$$
where $q=q_1+\cdots+q_p$.
(We have renumerated the pairs $(m_{ij},k_{ij})$ by a single index~$i$.)

Similarly, consider the algebra~$R_2$ and the ideal $$I_2=\left(\left.\frac{1}{(\lambda,\mu)}\,\right| \lambda,\mu\in \Lambda \text{ are primitive elements}, \lambda\neq \pm\mu\right).$$ There are two cases, $I_2=R_2$ and $I_2\neq R_2$. Again, $I_2\neq R_2$ contradicts Lemma~\ref{lem_main}. So we have $1\in I_2$. Then 
$$
1=\sum_{i=1}^{l} \frac{s_i}{(\lambda_{i1},\mu_{i1})(\lambda_{i2},\mu_{i2})\ldots (\lambda_{it_i},\mu_{it_i})},
$$
where $s_i\in R$, $t_i\geq 1$, and $\lambda_{ij}\neq\pm\mu_{ij}$ are primitive elements of $\Lambda$. Since $(\lambda_{ij},\mu_{ij})=A_{ij}g_{11}+B_{ij}g_{12}+C_{ij}g_{22}$, $A_{ij}, B_{ij}, C_{ij}\in\Z$, $B_{ij}^2\neq 4A_{ij}C_{ij}$, we obtain an equation $h(g_{11},g_{12},g_{22})=0$ such that

$$
\widehat{h}(X,Y,Z)=\prod_{j=1}^t (A_jX+B_jY+C_jZ), \ B_j^2\neq 4A_jC_j.
$$

It is clear that $\widehat{f}$ and $\widehat{h}$ are coprime. 
\end{proof}

\begin{proof}[Proof of Theorem \ref{main_t}]

If the set of numbers $l=\{l_{uv}\}$ cannot be realized as the set of the squares of the edge lengths of a polyhedral surface $(K,\theta)$, then the set $\mathcal{G}(K,\Lambda,l)$ is empty, hence, the assertion of the theorem is trivial. 

Now, suppose that $l$ is the set of the squares of the edge lengths of  $(K,\theta)$. Let $\eta \colon R\rightarrow \R$ be the restriction of the specialization homomorphism $\tau_{\theta}$ to $R$. Since $\eta(\ell_{uv})=l_{uv},$ the homomorphism $\eta$ is independent of the choice of $\theta$. Let $f,h\in R[X,Y,Z]$ be the polynomials in Proposition~\ref{propos_main}, and let $\overline{f}, \overline{h}\in \R[X,Y,Z]$ be their images under $\eta$. Since $\widehat{f}, \widehat{h}\in\Q[X,Y,Z]$ and $\eta|_\Q$ is the identity homomorphism,  the leading terms of $\overline{f}$ and $\overline{h}$ are again $\widehat{f}$ and $\widehat{h}$ respectively. But  $\widehat{f}$ and $\widehat{h}$ are coprime. Therefore, $\overline{f}$ and $\overline{h}$ are coprime. Hence the set $\mathcal{G}(K,\Lambda,l)$ is contained in the one-dimensional variety  
$$\left\{
\begin{aligned}
\overline{f}(X,Y,Z)&=0,\\
\overline{h}(X,Y,Z)&=0.
\end{aligned}
\right.
$$ 
\end{proof}

\section{Proof of Lemma~\ref{lem_main}}\label{section_proof_lemma}

Let $V$ be the vertex set of the simplicial complex~$K$. A vertex $v\in V$ will be called \textit{special\/} if it is connected by an edge of~$K$ with another vertex in the orbit~$\Lambda v$. If an edge~$e$ connects vertices~$v$ and~$T_{\lambda}(v)$, $\lambda\in\Lambda$, then  the edge  $T_{\mu}(e)$ connects the vertices~$T_{\mu}(v)$ and $T_{\lambda}(T_{\mu}(v))$ for every $\mu\in\Lambda$. Hence the property of being special is invariant under the action of~$\Lambda$.
 A $\Lambda$-orbit~$\Lambda v\subset V$ will be called \textit{special\/} if it consists of special vertices.

\begin{lem}\label{lem_primitive}
Suppose that a vertex $v$ of~$K$ is connected by an edge with a vertex~$T_{\lambda}(v)$, $\lambda\in\Lambda$. Then $\lambda$ is primitive.
\end{lem}
\begin{proof}
The quotient surface $K/\Lambda$ is homeomorphic to the torus. We have the natural isomorphism $H_1(K/\Lambda,\Z)\cong\Lambda$. We shall identify these two groups by this natural isomorphism. Let $e$ be an edge connecting~$v$ and~$T_{\lambda}(v)$, and
let $\gamma$ be the image of~$e$ under the quotient mapping $K\to K/\Lambda$.  The endpoints of~$e$ lie in the same $\Lambda$-orbit, and no other two points of~$e$ lie in the same $\Lambda$-orbit. Hence $\gamma$ is a simple closed curve of homology class~$\lambda$. Therefore, $\lambda$ is primitive.
\end{proof}

\begin{lem}\label{lem_special}
Suppose that $K$ contains at least one special vertex. Then there exists a primitive element $\lambda\in\Lambda$ such that every special vertex $v$ of~$K$ is connected by an edge with~$T_{\lambda}(v)$.
\end{lem}

\begin{proof}

Let $v$ be a special vertex. Then there exists an element~$\lambda\in\Lambda$ such that the vertices~$v$ and~$T_{\lambda}(v)$ are connected by an edge~$e$. By Lemma~\ref{lem_primitive}, $\lambda$ is primitive.

Let $v'$ be any other special vertex. If $v'=T_{\mu}(v)$ for a $\mu\in\Lambda$, then $v'$ is connected with $T_{\lambda}(v')$ by an edge~$T_{\mu}(e)$. Suppose that $v'\notin\Lambda v$.
Since $v'$ is special, there exists an element~$\lambda'\in\Lambda$ such that the vertices~$v'$ and~$T_{\lambda'}(v')$ are connected by an edge~$e'$. Let $\gamma$ and $\gamma'$ be the images of~$e$ and~$e'$ respectively under the quotient mapping $K\to K/\Lambda$. Then $\gamma$ and $\gamma'$ are simple closed curves of homology classes~$\lambda$ and~$\lambda'$ respectively. Since  $\Lambda v'\ne\Lambda v$, the simple closed curves~$\gamma$ and~$\gamma'$ are disjoint. Hence either $\lambda'=\lambda$ or $\lambda'=-\lambda$. Therefore, the vertices $v'$ and $T_{\lambda}(v')$ are joined either by the edge~$e'$ or by the edge~$T_{\lambda}(e')$.  
\end{proof}

We shall prove Lemma~\ref{lem_main} by induction on the number~$n$ of non-special $\Lambda$-orbits in~$V$.

\textsl{Base of induction.\/} Suppose that $n=0$. Then all vertices $v\in V$ are special. By Lemma~\ref{lem_special}, there exists a primitive element $\lambda\in\Lambda$ such that $v$ and $T_{\lambda}(v)$ are joined by an edge  for every~$v\in V$. We denote this edge by~$e_v$. 

The image of~$e_v$ under the quotient mapping $K\to K/\Lambda$ is a simple closed curve~$\gamma_{\Lambda\! v}$ of the homology class~$\lambda$. (Obviously, this curve depends  only on the $\Lambda$-orbit of~$v$.) For distinct $\Lambda$-orbits~$\Lambda v_1$ and~$\Lambda v_2$, the curves~$\gamma_{\Lambda\! v_1}$ and~$\gamma_{\Lambda\! v_2}$ are disjoint. Then the curves $\gamma_{\Lambda\! v}$ decompose the torus~$K/\Lambda$ into cylinders as shown in Figure~\ref{fig_torus}(a).
Let $\Lambda v_1,\ldots,\Lambda v_q$ be all different $\Lambda$-orbits in~$V$ numerated so that the curves $\gamma_i=\gamma_{\Lambda\! v_i}$ go successively around the torus~$K/\Lambda$. We denote by~$C_i$ the cylinder bounded by the curves~$\gamma_i$ and~$\gamma_{i+1}$. 

The quotient $K/\Lambda$ is an \textit{ideal triangulation\/} of the torus, i.\,e., a triangulation which is not necessarily a simplicial complex. Vertices of~$K/\Lambda$ are in one-to-one correspondence with orbits~$\Lambda v$, $v\in V$. The restriction of the triangulation~$K/\Lambda$ to the cylinder~$C_i$ is a triangulation with exactly two vertices~$\Lambda v_i$ and~$\Lambda v_{i+1}$ whose boundary consists of two edges~$\gamma_{i}$ and~$\gamma_{i+1}$. Standard calculation of the Euler characteristic yields that this triangulation has exactly $2$ faces and $4$ edges.
It is easy to see that such ideal triangulation of the cylinder is unique  up to an isomorphism.  Hence the triangulation~$K/\Lambda$ is isomorphic to the triangulation shown in Figure~\ref{fig_torus}(b).  

\begin{figure}
\unitlength=5mm
\begin{picture}(19,7.5)
\put(3.3,0){\textit{a}}
\put(15.3,0){\textit{b}}

\multiput(0,1.5)(12,0){2}{%
\begin{picture}(0,0)

\thicklines

\put(0,0){\line(1,0){7}}
\put(0,6){\line(1,0){7}}
\put(0,0){\vector(1,0){3.8}}
\put(0,6){\vector(1,0){3.8}}
\put(0,0){\line(0,1){6}}
\put(7,0){\line(0,1){6}}
\put(0,0){\vector(0,1){3.48}}
\put(7,0){\vector(0,1){3.48}}
\put(0,0){\vector(0,1){3.12}}
\put(7,0){\vector(0,1){3.12}}

\thinlines

\multiput(1.7,0)(1.2,0){4}{\line(0,1){6}}
\multiput(1.7,3)(1.2,0){4}{\circle*{.15}}

\end{picture}%
}

\put(0,1.5){%
\begin{picture}(0,0)
\multiput(.45,2.95)(5.45,0){2}{$\ldots$}
\footnotesize
\put(1.75,2.55){$\Lambda\hspace{-.2mm}v\hspace{-.2mm}_q$}
\put(2.95,2.55){$\Lambda\hspace{-.2mm}v\hspace{-.2mm}_1$}
\put(4.15,2.55){$\Lambda\hspace{-.2mm}v\hspace{-.2mm}_2$}
\put(5.35,2.55){$\Lambda\hspace{-.2mm}v\hspace{-.2mm}_3$}

\put(1.75,4.5){$\gamma\hspace{-.2mm}_q$}
\put(2.95,4.5){$\gamma\hspace{-.2mm}_1$}
\put(4.15,4.5){$\gamma\hspace{-.2mm}_2$}
\put(5.35,4.5){$\gamma\hspace{-.2mm}_3$}

\put(2,1){$C\hspace{-.2mm}_q$}
\put(3.2,1){$C\hspace{-.2mm}_1$}
\put(4.4,1){$C\hspace{-.2mm}_2$}

\end{picture}%
}

\put(12,1.5){%
\begin{picture}(0,0)
\multiput(.45,2.95)(5.3,0){2}{$\ldots$}
\put(1.7,3){\line(1,0){3.6}}
\multiput(1.7,3)(1.2,0){3}{\line(1,-5){.6}}
\multiput(2.3,6)(1.2,0){3}{\line(1,-5){.6}}
\end{picture}%
}

\end{picture}
\caption{The torus~$K/\Lambda$: (a) decomposition into cylinders, (b) triangulation}\label{fig_torus}
\bigskip
\bigskip

\input{fig_K.tex}
\caption{The triangulation~$K$}\label{fig_K}
\end{figure}

Therefore, the triangulation~$K$ is isomorphic to the simplicial complex shown in Figure~\ref{fig_K}.
The pre-image under the quotient mapping $K\to K/\Lambda$ of every curve~$\gamma_i$ is a union of countably many lines $L_{i+qj}$, $j\in\Z$, each homeomorphic to~$\R$.  The pre-image of every cylinder~$C_i$ is a union of countably many strips~$S_{i+qj}$,  $j\in\Z$, each homeomorphic to~$\R\times I$. The vertices $v_i\in L_i$ may be chosen so that the vertices~$v_i$, $v_i'=T_{\lambda}(v_i)$, and~$v_{i+1}$ span a triangle in~$S_i$ for every~$i$. The line~$L_i$ consists of  the vertices~$T_{\lambda}^k (v_i)$ and the edges~$T_{\lambda}^k (e_{v_i})$, $k\in\Z$. 

The orbit $\Lambda v_i$ is the union of the vertex sets of the lines~$L_{i+qj}$, $j\in\Z$. Let $v_q=T_{\mu}(v_0)$. Then $v_{i+q}=T_{\mu}(v_{i})$ for every $i$. It easily follows that any vertex in the orbit~$\Lambda v_i$ has the form $T_{\lambda}^kT_{\mu}^m(v_i)=T_{k\lambda+m\mu}(v_i)$. Therefore, $\lambda,\mu$ is a basis of~$\Lambda$.
We have  $(\lambda,\lambda)=\ell_{v_0v_0'}$ and
\begin{equation*}
(\lambda,\mu)=(\lambda,v_q-v_0)=\sum_{i=0}^{q-1}(v_i'-v_i^{\vphantom{\prime}},v_{i+1}^{\vphantom{\prime}}-v_i^{\vphantom{\prime}})=\frac{1}{2}\sum_{i=0}^{q-1}(\ell_{v_i^{\vphantom{\prime}}v_i'}+\ell_{v_i^{\vphantom{\prime}}v_{i+1}^{\vphantom{\prime}}}-\ell_{v_i'v_{i+1}^{\vphantom{\prime}}}).
\end{equation*}
Since the place~$\varphi$ is finite on $\ell_{vw}$ for all edges~$[vw]$ of~$K$, and $\mathop{\mathrm{char}}\nolimits\F\ne 2$, we conclude that $\varphi$ is finite on~$(\lambda,\lambda)$ and~$(\lambda,\mu)$.

\textsl{Step of induction.\/} Let $n$ be the number of non-special $\Lambda$-orbits in~$V$. Let $d$ be the smallest degree of a non-special vertex in~$V$. (The \textit{degree\/} of a vertex is the number of edges incident to it.) We assume that we have already proved the assertion of Lemma~\ref{lem_main} for all~$K$ with strictly smaller~$n$, or with the same~$n$ and strictly smaller~$d$.

We shall consider two cases.

1. Suppose that $K$ contains an ``empty triangle'', that is, three vertices $u$, $v$, and~$w$ such that $[uv]$, $[vw]$, and $[wu]$ are edges of~$K$, but $[uvw]$ is not a triangle of~$K$. Since $K$ is homeomorphic to a plane, we see that the curve~$\gamma$ consisting of the edges $[uv]$, $[vw]$, and $[wu]$ decomposes~$K$ into two parts, one of which is homeomorphic to a disk. We denote the closure of this part by~$D$. Then~$D$ is a finite subcomplex of~$K$. 

\begin{lem}\label{lem_interiorsD}
The interiors of the disks~$D_{\lambda}=T_{\lambda}(D)$, $\lambda\in\Lambda$, are disjoint. 
\end{lem}
\begin{proof}
Obviously, it is sufficient to prove that the interiors of the disks~$D$ and~$D_{\lambda}$ are disjoint whenever $\lambda\ne 0$. If $D\supset D_{\lambda}$, then $D\supset D_{k\lambda}$ for all $k>0$. Hence, $D$ contains infinitely many distinct vertices~$T_{k\lambda}(u)$, which is impossible, since $D$ is a finite simplicial complex. Therefore, $D\not\supset D_{\lambda}$. Similarly, $D\not\subset D_{\lambda}$.  Suppose, $\Int(D)\cap\Int(D_{\lambda})\ne\emptyset$; then the intersection $\partial D\cap\partial D_{\lambda}$ must contain at least two distinct points. But this intersection is a subcomplex of~$K$. Hence it contains at least two distinct vertices of~$K$. This means that two of the three vertices $T_{\lambda}(u)$, $T_{\lambda}(v)$, and $T_{\lambda}(w)$ coincide with two of the three vertices~$u$, $v$, and~$w$. Since the action of~$\Lambda$ on~$K$ is free, we obtain that $T_{\lambda}(s)\ne s$ for all vertices~$s$, and, if~$T_{\lambda}(s)=t$, then $T_{\lambda}(t)\ne s$. It is not hard to check that up to a  permutation of~$u$, $v$, and~$w$, there is a unique possibility $T_{\lambda}(u)=v$ and $T_{\lambda}(v)=w$. Then the vertices $u$ and $w=T_{2\lambda}(u)$ are connected by an edge in~$K$. But this is impossible by Lemma~\ref{lem_primitive}. Thus, the interiors of~$D$ and~$D_{\lambda}$ are disjoint.
\end{proof}

Since $[uvw]$ is not a triangle of~$K$, we see that the interior of~$D$ contains at least one vertex~$p$. Lemma~\ref{lem_interiorsD} implies that the vertex~$p$ is non-special. Consider the simplicial complex~$K_1$ obtained from~$K$ by replacing every subcomplex~$T_{\lambda}(D)$ by the triangle~$[T_{\lambda}(u)T_{\lambda}(v)T_{\lambda}(w)]$. It is easy to see that $K_1$ is a well-defined simplicial complex that is homeomorphic to a plane and periodic with the same period lattice~$\Lambda$.  Let $V_1$ be the vertex set of~$K_1$. The number of non-special $\Lambda$-orbits in~$V_1$ is strictly less than the number of non-special $\Lambda$-orbits in~$V$, since at least one non-special orbit $\Lambda p$ has been deleted. By the inductive assumption, the assertion of Lemma~\ref{lem_main} holds for $K_1$.  Since all edges of~$K_1$ are edges of~$K$, this immediately implies the assertion of Lemma~\ref{lem_main} for~$K$.

2. Suppose that $K$ contains no ``empty triangles''. We shall use the following lemma due to Connelly, Sabitov, and Walz~\cite{CSW97}. 

\begin{lem}\label{lem_CSW}
Let $u$ be a vertex of a $2$-dimensional simplicial manifold~$K,$ and let $v_1,\ldots,v_d,$ $d\ge 4,$ be all vertices  joined by edges with~$u,$ numbered in succession around~$u$. Let $\varphi$ be a place that is defined on the field $\K=\Q(x_u,y_u,z_u,x_{v_1},y_{v_1},z_{v_1},\ldots,x_{v_d},y_{v_d},z_{v_d})$ and is finite on all squares of the edge lengths~$\ell_{uv_i}$, $i=1,\ldots,d$, $\ell_{v_iv_{i+1}}$, $i=1,\ldots,d-1$, and $\ell_{v_dv_1}$. Then $\varphi$ is finite on at least one of the squares of the diagonal lengths $\ell_{v_iv_{i+2}},$ $i=1,\ldots,d-2$.
\end{lem}

We shall take for $u$ a non-special vertex of~$K$ of the smallest degree~$d$, and define $v_1,\ldots,v_d$ as in the above lemma.  Since $K$ does not contain ``empty triangles'', we see that $d\ge 4$, and $[v_iv_{i+2}]$ are not edges of~$K$, $i=1,\ldots,d-2$. We cannot apply Lemma~\ref{lem_CSW} immediately to our place $\varphi\colon\L\to\F\cup\{\infty\}$, since the field $\K$ is not necessarily a subfield of~$\L$. Indeed, some of the vertices $v_j$ may belong to the same $\Lambda$-orbits, hence, their coordinates may be dependent in~$\L$. Nevertheless, we can consider the composite place 
$$
\tilde\varphi\colon\K\xrightarrow{\psi}\L\cup\{\infty\}\xrightarrow{\varphi}\F\cup\{\infty\},
$$
where $\psi$ is the place taking independent variables $x_u,y_u,\ldots,z_{v_p}$ to the elements of~$\L$ denoted by the same letters. Applying Lemma~\ref{lem_CSW} to the place~$\tilde{\varphi}$, we obtain that there is an~$i$ such that $\tilde{\varphi}(\ell_{v_iv_{i+2}})\ne\infty$. Then $\varphi(\ell_{v_iv_{i+2}})\ne\infty$.

We replace the two triangles $[uv_iv_{i+1}]$ and~$[uv_{i+1}v_{i+2}]$ by the two triangles $[uv_iv_{i+2}]$ and $[v_iv_{i+1}v_{i+2}]$. Since $[v_iv_{i+2}]$ is not an edge of~$K$, we obtain a simplicial complex homeomorphic to a plane. To keep the simplicial  complex $\Lambda$-periodic we simultaneously replace every two triangles $T_{\lambda}([uv_iv_{i+1}])$ and~$T_{\lambda}([uv_{i+1}v_{i+2}])$ by the two triangles $T_{\lambda}([uv_iv_{i+2}])$ and $T_{\lambda}([v_iv_{i+1}v_{i+2}])$. (All triangles $T_{\lambda}([uv_iv_{i+1}])$ and~$T_{\lambda}([uv_{i+1}v_{i+2}])$, $\lambda\in\Lambda$, are pairwise distinct. Hence we actually can perform all described flips simultaneously.)  We denote the resulting simplicial complex by~$K'$. All edges of~$K'$ are edges of~$K$ except for the edges $T_{\lambda}([v_iv_{i+2}])$. Since $\varphi(\ell_{v_iv_{i+2}})\ne\infty$, we see that $\varphi$ is finite on all squares of the edge lengths of~$K'$. 

The vertex set of~$K'$ coincides with the vertex set~$V$ of~$K$. An edge of~$K$ is not an edge of~$K'$ if and only if it coincides with one of the edges $T_{\lambda}([uv_{i+1}])$. Since $u$ is non-special, any such edge connects vertices belonging to distinct $\Lambda$-orbits. Hence any vertex $w$ that is special in~$K$ is also special in~$K'$. Since $u$ is non-special in~$K$, we see that none of the vertices~$T_{\lambda}(v_j)$, $\lambda\in\Lambda$, $j=1,\ldots,d$, coincide with~$u$. Hence the degree of~$u$ in~$K'$ is equal to~$d-1$ and $u$ is non-special in~$K'$. Therefore, the number of non-special $\Lambda$-orbits of vertices of~$K'$ is not greater than~$n$, and the smallest degree of a non-special vertex of~$K'$ is strictly less than~$d$. Applying the inductive assumption for~$K'$, we obtain that there exists a basis $\lambda,\mu$ of~$\Lambda$ such that $\varphi$ is finite on~$(\lambda,\lambda)$ and $(\lambda,\mu)$, which completes the proof of Lemma~\ref{lem_main}.

\end{document}